\documentclass[11pt]{article}
\usepackage{authblk}
\title{An analogue of Pillai's theorem for continued fraction
  normality and an application to subsequences}
\author[1]{Satyadev Nandakumar}
\author[1]{Subin Pulari}
\author[1]{Prateek Vishnoi}
\author[1]{Gopal Viswanathan}
\affil[1]{
  Department of Computer Science and Engineering\\
  Indian Institute of Technology Kanpur,
  Kanpur, Uttar Pradesh, India.
}
\affil[]{\{\textit{satyadev,subinp,pratvish,gopalv}\}@cse.iitk.ac.in}

\usepackage{amsmath,amsthm,amssymb}
\usepackage[margin=1in]{geometry}
\usepackage[colorlinks=true]{hyperref}

\usepackage{amsfonts}
\usepackage{amsmath} 
\usepackage{amssymb}
\usepackage{enumerate}

\newcommand{\PROOF}{\begin{proof}}

\newcommand{\QED}{\end{proof}}

\newcommand{\N}{\mathbb{N}}

\newcommand{\R}{\mathbb{R}}

\usepackage[svgnames]{xcolor}

\newcommand{\denominator}{\text{denominator}}
\newcommand{\numerator}{\text{numerator}}
\theoremstyle{plain}
\newtheorem{theorem}{Theorem}[section]
\newtheorem{lemma}[theorem]{Lemma}

\theoremstyle{definition}

\newtheorem{definition}[theorem]{Definition}

\begin{document}

\maketitle

\begin{abstract}
  We show that two notions of continued fraction normality, one where
  overlapping occurrences of finite patterns are counted as distinct
  occurrences, and another where only disjoint occurrences are counted
  as distinct, are identical. This equivalence involves an analogue of
  a theorem due to S. S. Pillai \cite{Pillai40} for base-$b$
  expansions. The proof requires techniques which are fundamentally
  different, since the continued fraction expansion utilizes a
  countably infinite alphabet, leading to a non-compact space.

  Utilizing the equivalence of these two notions, we provide a new
  proof of Heersink and Vandehey's recent result that selection of
  subsequences along arithmetic progressions does not preserve
  continued fraction normality \cite{HeersinkVandehey16}.
\end{abstract}

\section{Introduction}
Any irrational real number in $x \in (0,1)$ has a unique continued
fraction expansion
\begin{align*}
x &= 0 + \cfrac{1}{a_1+\cfrac{1}{a_2 + \cfrac{1}{a_3 +
			\cfrac{1}{\ddots}}}}, 
\end{align*}
where $a_i \in \N^+$, which we denote by the sequence
$[0;a_1,a_2,\dots]$. Since we work exclusively in the unit interval,
we ignore the integer part of the number in the following
discussion. A continued fraction is said to be \emph{normal} if the
asymptotic frequency of any block of integers $w$ is equal to
$\gamma(C_w)$, where $\gamma$ is the Gauss measure and $C_w$ is the
set of all continued fractions with $w$ as the prefix (definitions are
provided in the following section.) In this notion, when we count
occurrences of $w$, we count overlapping occurrences of $w$ as
distinct occurrences.

Another natural notion of normality counts only disjoint occurrences
of any block $w$ as distinct occurrences, and requires the asymptotic
disjoint frequency to be $\gamma(C_w)$. In combinatorial arguments,
this latter notion is easier to study.

It is not obvious that these notions are equivalent. In the base-$b$
expansion of reals in the unit interval, where $b \in \N$ and
$b \ge 2$, these two notions, namely, normality via overlapping
frequencies and via disjoint frequencies, coincide. The fact that
overlapping block normality implies disjoint block normality follows
from the Postnikov--Piateskii-Shapiro criterion\footnote{also known as
	the ``hot-spot
	theorem''}\cite{Postnikov60},\cite{PostnikovPiatetskii57}, which is
also known to hold for the continued fraction expansion. The fact that
disjoint occurrence normality implies overlapping block normality is a
consequence of Pillai's Theorem \cite{Pillai40}. The proof of Pillai's
theorem relies on properties of the finite alphabet used in the
base-$b$ expansions.  It is not immediately clear whether this
generalizes to continued fraction expansions.

We establish that these notions indeed coincide in the case of
continued fraction normality. The proofs employ very different
techniques from the classical base-$b$ expansions, since the set of
entries allowed at any position in the expansion is countably
infinite, leading to a non-compact topological space. We prove this
result for continued fractions with a technique which does not require
compactness.

Further, we demonstrate the utility of our result by giving an
alternative proof of a recent result by Heersink and Vandehey
\cite{HeersinkVandehey16} that if $[0;a_1,a_2,\dots]$ is a continued
fraction normal, and $(i_1, i_2, \dots)$ is any non-trivial arithmetic
progression, then $[0;a_{i_1},a_{i_2},\dots]$ is \emph{not} continued
fraction normal. This result contrasts with the behavior of base-$b$
normality, where such a subsequence is always normal
\cite{Wall49}. Heersink and Vandehey employ an ergodic-theoretic
argument whereas our proof is combinatorial.

\section{Preliminaries}

Let $\N$ denote the set of natural numbers $\{0,1,\dots\}$, and for a
non-negative integer $k$, let $\N_{>k}$ represent the set of numbers
greater than $k$. For a finite sequence of numbers $w$, we denote its
length by $|w|$, and the $i^{\text{th}}$ coordinate of $w$ is denoted
by $w_i$, where $1 \le i \le |w|$. The set of finite sequences of
numbers of length $m$ are represented by $\N^m$, and the set of
infinite sequences of numbers, by $\N^\infty$.

We consider continued fraction expansions of the form
$[0;a_1, a_2, \dots]$ where $a_i \in \N_{>0}$. For a finite sequence
of numbers $w$, we say that $w$ is a prefix of the (finite or
infinite) continued fraction $[0;a_1, a_2, \dots]$ if $a_i=w_i$ for
$1 \le i \le |w|$.  If $w$ is a prefix of the continued fraction
expansion of $x$, we write $w \sqsubseteq x$.

We work in the probability space $(\N^\infty, \mathcal{B}, \gamma)$
where $\mathcal{B}$ is the Borel $\sigma$-algebra generated by the
cylinders of the form
$$C_w = \{ x \in [0,1] \mid w \sqsubseteq x \},$$
and $\gamma$ is the Gauss measure defined by
\begin{align}
\gamma(A) = \frac{1}{\ln 2} \int_A \frac{1}{1+x} ~dx, \qquad A \in
\mathcal{B}. 
\end{align}
If $S=\{w_1,w_2,\dots\}$ is a set of strings, then for any $x \in (0,1)$, by $x \in C_{S}$
we mean that $x \in C_{w_i}$ for some $i \in \N_{>0}$.

To study the block statistics of continued fraction normals, we
introduce the left-shift transformation $T: \N^\infty \to \N^\infty$
defined by
$$T([0;a_1,a_2,\dots]) = [0;a_2,a_3,\dots]$$
which can be identified with the transformation $\hat{T}:(0,1) \to
(0,1)$ defined by
$$\hat{T}(x) = \frac{1}{x} \mod 1 = \frac{1}{x} - \left\lfloor
\frac{1}{x} \right\rfloor,$$ where $\lfloor y \rfloor$ is the
greatest integer lesser than or equal to $y$. The left-shift
transformation $T$ is strongly mixing, hence \emph{a fortiori} ergodic
and measure-preserving, with respect to the Gauss measure.

\begin{definition}
	A real number $x \in [0,1]$ with continued fraction expansion $[0;
	a_1, a_2, \dots]$ is said to be \emph{continued fraction normal} if
	for every $w \in \N^*$,
	\begin{equation}
	\lim_{n \to \infty} \frac{|\{i \mid T^ix \in
		C_w , 0 \le i \le n  \}|}{n} = \gamma(C_w).
	\end{equation}
\end{definition}

In other words, a real number $x$ is continued fraction normal if the
asymptotic frequency of all blocks of integers $w$ is equal to the
Gauss measure of the cylinder $C_w$. Note that $T^i x \in C_w$ if and
only if $w$ is a prefix of the continued fraction expansion of $x$
(mod 1). In this notion, we consider overlapping occurrences of $w$ as
distinct - for example, if $w$ is $[1,1]$, then it occurs twice in the
subsequence $[1,1,1]$. Since $T$ is ergodic with respect to the Gauss
measure, by the ergodic theorem, it follows that the set of continued
fraction normals has Gauss measure 1.\footnote{Since the Gauss measure
	is absolutely continuous with respect to the Lebesgue measure, the
	set of continued fraction normals has Lebesgue measure 1 as well.}

\section{An Analogue of Pillai's theorem}

In this section, we prove the analogue of Pillai's theorem for
continued fractions.  Let $x \in (0,1)$.

Suppose that the limiting frequency of disjoint blocks of any finite
length string $\mathnormal{s}= a_1a_2...a_k$,
$a_i \in \mathbb{N}, k \in \mathbb{N}$ in the continued fraction
expansion of $\mathnormal{x}$ exists, and is equal to the Gauss
measure of the cylinder set of $\mathnormal{s}$, \emph{i.e.}
\begin{align}
\lim_{n\to\infty} \frac{|\{i \mid T^{ki}x \in C_s , 0 \leq i \leq
	\frac{n}{k} \}|} {n/k} = \gamma(C_s).
\end{align}

The analogue of Pillai's theorem for continued fractions is the
assertion that the limiting sliding block frequency of any finite
length string $s$ of positive integers during a long run is equal to
the Gauss measure of its cylinder set.

\begin{theorem}[Analogue of Pillai's Theorem for Continued Fractions]
	\label{th:pillaistheorem}
	Let $\mathnormal{x} \in (0,1)$. If for all 
	$k \in \mathbb{N}$, and for all strings $s \in \mathbb{N}^k_{>0}$, if we
	have
	\begin{align}
	\lim_{n\to\infty} \frac{|\{i \mid T^{ki}x \in C_s , 0 \leq i \leq
		\frac{n}{k} \}|} {n/k} =
	\gamma(C_s),
	\end{align}
	then, for any finite string $s \in \N^*_{>0}$,
	\begin{align}
	\lim_{n\to\infty} \frac{|\{i \mid T^{i}x \in C_s , 0 \leq i \leq n
		\}|} {n} = \gamma(C_s).
	\end{align}
\end{theorem}
The proof of the above theorem requires the following lemma.  This
estimates, for any real with the right frequency of disjoint
$k$-length strings, the overlapping frequency of $k$-length strings
inside longer blocks. The proof of this lemma requires new techniques
which can deal with non-compact spaces, and carry out combinatorial
arguments where the number of possible patterns of any given length
is infinite.
\begin{lemma}
	\label{lem:limitlemma}
	Let $\mathnormal{x} \in (0,1)$. Suppose that for all
	$k \in \N_{>0}$, and all $s\in\mathbb{N}^k_{>0}$,
	\begin{align*}
	\lim_{n\to\infty} \frac{|\{i \mid T^{ki}x \in C_s , 0 \leq i \leq
		\frac{n}{k} \}|} {n/k} = \gamma(C_s).
	\end{align*}
	Let $s' \in \N^k_{>0}$ be any finite string and $p,q\in\mathbb{N}_{>0}$. Let $ D_1, D_2,..$
	denote any enumeration of all $(p+q+k)$ length strings such that
	$s'$ occurs in it at the position $(p+1)$.  \emph{i.e.}, for any
	$D_{j}=b_{1}b_{2}\dots b_{p+q+k}$ we have
	$b_{p+1}b_{p+2}\dots b_{p+1+k}=s'$. Then the following holds.
	\begin{align}
	\label{eqn:lemma}
	\lim_{n\to\infty}\sum_{j =1}^{\infty} \frac{|\{i \mid T^{(p+q+k)i}x \in
		C_{D_j} , 0 \leq i \leq n / (p+q+k)    \}|}{n/(p+q+k)} =
	\gamma(C_{s'})  
	\end{align}
\end{lemma}

We first assume the above lemma and prove Theorem
\ref{th:pillaistheorem}. A proof of Lemma \ref{lem:limitlemma} is
given at the end of this section. The proof of Theorem
\ref{th:pillaistheorem} given below is an adaptation of the technique
employed in \cite{maxfield1952short} by John E. Maxfield in the
setting of finite alphabet, to our setting of countably infinite
alphabet.

\begin{proof}[Proof of Theorem \ref{th:pillaistheorem}]
	We know that the sliding block frequency of any finite length string
	$s=a_1a_2a_3\dots a_k\in \N^k_{>0}$ for large enough $n$ can be written as:
	\begin{align}
	\label{eq:pillaistheorem_one}
	\frac{|\{i \mid T^{i}x \in C_s , 0 \leq i \leq (n-k) \}|} {n} =
	f_1(n) + f_2(n) +f_3(n) \dots
	f_{(1+\left \lfloor\log_2  \frac{n}{k} \right \rfloor)}(n) + {\frac{o(n)}{n}}
	\end{align}
	where $f_p(n)$ calculates the frequency of $s$ in blocks of length
	$2^{p-1}k$ that are not counted in previous functions (i.e
	$f_1(n)$,$f_2(n)$,$\dots$,$f_{p-1}(n)$), and is defined by
	\begin{align*}
	f_{p}(n)= \begin{cases}
	\sum_{j=1}^{k-1} \frac{|\{i~ \mid~ T^{(2^{p-1})ki}x ~\in~ C_{S_j}
		~	,~ 0 ~\leq ~i ~\leq ~n/2^{p-1}k \}|} {n},\text{ if } p \leq
	1+\left \lfloor\log_2(n/k)\right \rfloor\\ 
	0, \text{ otherwise}
	\end{cases}
	\end{align*}
	where $S_j$ is an infinite collection of $2^{(p-1)} k$ length blocks
	s.t $s$ occurs in it at starting position $(2^{(p-2)}k - j +1)^{th}$
	position i.e $S_{j}$ is the set of strings of the form,
	$u\ a_1a_2 \dots a_k\ v$ where $u$ is some string of $2^{p-2}k-j$
	positive integers, and $v$ is some string of $2^{p-2}k-k+j$ positive
	integers.
	
	Taking limits on both sides of (\ref{eq:pillaistheorem_one}) we get,
	$$ \lim_{n\to\infty} \frac{|\{i \mid T^{i}x \in C_s , 0 \leq i \leq n \}|}
	{n} = \lim_{n\to\infty} \sum_{i=1}^{\infty} f_{i}(n).$$ We know show
	that the sequence $\langle \sum_{i=1}^{m} f_i(n) \rangle_{m \in \N}$
	is uniformly convergent. This will enable us to interchange the limit
	and the summation in the previous equation.
	
	Fix $\varepsilon>0$ arbitrarily small. We show that there is a $t$
	such that for all $n$, $| \sum_{i =t}^{\infty} f_i(n)| < \epsilon$.
	Now
	$$f_t(n) \leq \frac{n / 2^{t-1}k}{n}.(k-1) =
	\Big(1-\frac{1}{k}\Big)\frac{1}{2^{t-1}}$$
	Hence
	$$ \sum_{i =t}^{\infty}f_t(n) \leq \sum_{i
		=t}^{\infty}\Big(1-\frac{1}{k}\Big)\frac{1}{2^{t-1}} =
	\Big(1-\frac{1}{k}\Big) \frac{1}{2^{t-2}}$$ 
	We can make the above term smaller as much as we need by choosing a
	large $t$. Thus the sequence $\langle \sum_{i=1}^{m} f_i(n) \rangle_{m
		\in \N}$ is uniformly convergent.

	So,
	$$ \lim_{n\to\infty} \frac{|\{i \mid T^{i}x \in C_s , 0 \leq i \leq n \}|}
	{n} = \sum_{i=1}^{\infty} \lim_{n\to\infty} f_i(n).$$
	
	Now it is given that,
	\begin{align}
	\label{eq1}
	\lim_{n\to\infty} \frac{|\{i \mid T^{ki}x \in C_s , 0 \leq i \leq
		\frac{n}{k} \}|} {n/k} = \gamma(C_s) 
	\end{align}	
	and thus,	
	$$ \lim_{n\to\infty} f_1(n) = \lim_{n\to\infty}\frac{|\{i \mid T^{ki}x \in
		C_s , 0 \leq i \leq n/k \}|} {n} =
	\frac{\gamma(C_s)}{k}. $$
	Now, when $p \leq 1+ \left\lfloor\log_2(\frac{n}{k}) \right\rfloor$,
	\begin{align*}
	\lim_{n\to\infty} f_p(n)
	=& \lim_{n\to\infty} \sum_{j =1}^{k-1}
	\frac{|\{i \mid T^{({2^{p-1})}ki}x \in C_{S_j} , 0 \leq i \leq n/2^{p-1}k
		\}|} {n}\\	
	=&  \sum_{j =1}^{k-1}\lim_{n\to\infty}
	\frac{|\{i \mid T^{2^{p-1}ki}x \in C_{S_j} , 0 \leq i \leq n/2^{p-1}k \}|}
	{n}.
	\end{align*}
	The interchange between the sum and the limit is valid since there are
	only finitely many terms in the sum.
	
	Now we derive an expression for 
	$$ \lim_{n\to\infty} \frac{|\{i \mid T^{2^{p-1}ki}x \in C_{S_j} , 0 \leq i
		\leq n/2^{p-1}k \}|} {n}.  $$ 
	\\Let $D_{1},D_{2},D_{3}\dots$ be any enumeration of $S_{j}$. Now observe that,
	$$
	\sum_{m=1}^{\infty} |\{i \mid T^{2^{p-1}ki}x \in C_{D_{m}} , 0 \leq i \leq
	n/2^{p-1}k \}|
	= |\{i \mid T^{2^{p-1}ki}x \in C_{S_j} , 0 \leq i \leq n/2^{p-1}k \}|.
	$$
	In the proof by Maxfield, the expression on the right side is
	calculated directly by counting. Since we have a countably infinite
	alphabet, we have to adopt a different approach at this step. Since
	equation (\ref{eq1}) holds, applying Lemma
	\ref{lem:limitlemma},  we have,
	\begin{align*}
	\lim_{n\to\infty} \frac{|\{i \mid T^{2^{p-1}ki}x \in C_{S_j} , 0 \leq i
		\leq n/2^{p-1}k \}|} {n/2^{p-1}k} &=  	\lim_{n\to\infty}
	\sum_{m=1}^{\infty}\frac{
		|\{i \mid T^{2^{p-1}ki}x \in C_{D_{m}}
		, 0 \leq i \leq n/2^{p-1}k \}|}
	{n/2^{p-1}k} 
	\\ &=\gamma(C_{s}) 
	\end{align*}
	And therefore we get that
	$$
	\lim_{n\to\infty} \frac{|\{i \mid T^{2^{p-1}ki}x \in C_{S_j} , 0 \leq i \leq
		n/2^{p-1}k \}|} {n} = \frac{\gamma(C_{s})}{2^{p-1}k} 
	$$
	Thus,
	$$ \lim_{n\to\infty} f_p(n) = \lim_{n\to\infty} \sum_{j=1}^{k-1}
	\frac{|\{i \mid T^{2^{p-1}ki}x \in C_{S_j} , 0 \leq i \leq n/2^{p-1}k \}|}
	{n} = \frac{\gamma(C_s)}{2^{p-1}k}.(k-1)$$  
	Summing over all $f_i$, we obtain
	\begin{align*}
	\lim_{n\to\infty} \frac{|\{i \mid T^{i}x \in C_s , 0 \leq i \leq n \}|} {n}
	&= \sum_{i =1}^{\infty} \lim_{n\to\infty} f_i(n)\\
	&= \frac{\gamma(C_s)}{k} + (k-1)\gamma(C_s) \Big[\frac{1}{2k} +
	\frac{1}{2^2k}+ \dots \Big] \\
	&= \gamma(C_s),
	\end{align*}
	establishing the result.
\end{proof}

We now prove Lemma \ref{lem:limitlemma}, emphasizing our means of
dealing with countably infinite alphabets. 

\begin{proof}[Proof of Lemma \ref{lem:limitlemma}]
	Assume that $S$ denotes the collection of all $p+q+k$ length strings
	of natural numbers such that $s'$ occurs in it at starting at
	position $p+1$. We know that $S$ is a countably infinite set.
	Assume $\langle D_1, D_2, \dots \rangle$ be an enumeration of the
	elements of $S$.  Now,
	\begin{align}
	\label{eqn:1}
	\sum_{j = 1}^{\infty} \gamma(C_{D_{j}}) = \gamma(C_{s'})
	\end{align} 
	holds true since the Gauss transformation is measure preserving (see
	\cite{einsiedler2013ergodic}).
	
	First we show that for any $\epsilon > 0$, there are $m$ and $m'$
	such that for all $n \geq m'$, we have the following.
	\begin{align}
	\label{inequality2}
	\sum_{j = 1}^{m}
	\frac{|\{i \mid T^{(p+q+k)i}x \in C_{D_j} ,
		0 \leq i \leq \frac{n}{p+q+k}\}|}
	{n/(p+q+k)} >  \gamma(C_{s'}) - 2\epsilon.
	\end{align}
	
	Given $\epsilon>0$, due to equation (\ref{eqn:1}), we choose an $m$
	such that
	\begin{align}
	\label{inequality1}
	\sum_{j = 1}^{m} \gamma(C_{D_j}) > \gamma(C_{s'}) - \epsilon. 	
	\end{align}
	
	Now for every $D_j \in \langle D_1, D_2,....D_m \rangle$, by
	assumption, there is an $n_j$ such that for every $n \geq n_j$, we
	have
	$$\frac{|\{i \mid T^{(p+q+k)i}x \in C_{D_j} , 0 \leq i \leq \frac{n}{p+q+k}
		\}|}{n/(p+q+k)} > \gamma(C_{D_j}) - \frac{\epsilon}{2^j}. $$
	
	Let $m'$ = $\max \{n_1, n_2,...n_m \}$. Now, for every
	$D_j \in \langle D_1, D_2,....D_m\rangle$ and for every $n \geq m'$,
	we have,
	$$  \frac{|\{i \mid T^{(p+q+k)i}x \in C_{D_j}, 0 \leq i \leq \frac{n}{p+q+k}
		\}|}{n/(p+q+k)} >  \gamma(C_{D_j}) - \frac{\epsilon}{2^j} $$ 
	which implies that
	$$ \sum_{j = 1}^{m} \frac{|\{i \mid T^{(p+q+k)i}x \in C_{D_j} , 0 \leq i
		\leq \frac{n}{p+q+k} \}|}{n/(p+q+k)} > \sum_{j = 1}^{m}\gamma(C_{D_j})
	- \sum_{j = 1}^{m}\frac{\epsilon}{2^j} $$ From the inequality
	(\ref{inequality1}) we get that there exist $m, m',$ such that 
	for all $n \geq m'$
	\begin{align}
	\sum_{j = 1}^{m} \frac{|\{i \mid T^{(p+q+k)i}x \in C_{D_j} , 0 \leq i
		\leq \frac{n}{p+q+k}    \}|}{n/(p+q+k)} >  \gamma(C_{s'}) - 2\epsilon.
	\end{align}
	This establishes inequality (\ref{inequality2}).
	
	Now, similarly, we can obtain an upper bound. \emph{i.e.}, for every
	$\epsilon > 0$, there exist $m$, $m'$ such that for every $n > m'$,
	\begin{align}
	\label{inequality4}
	\sum_{j = 1}^{m} \frac{|\{i \mid T^{(p+q+k)i}x \in C_{D_j} , 0 \leq i
		\leq \frac{n}{p+q+k}    \}|}{n/(p+q+k)} < \gamma(C_{s'}) + 2\epsilon.  
	\end{align}
	
	Now let $\langle E_1, E_2, \dots\rangle$ be an enumeration of
	$(p+q+k)$ length strings of natural numbers such that $s'$ doesn't
	occurs in it at starting position $(p+1)$. Now, the following is a
	trivial observation:
	\begin{equation}
	\label{eqn:2}
	\begin{split}
	\sum_{j =1}^{\infty} \frac{|\{i \mid T^{(p+q+k)i}x \in C_{D_j} , 0
		\leq i \leq \frac{n}{p+q+k}    \}|}{n/(p+q+k)} &+  \sum_{j
		=1}^{\infty} \frac{|\{i \mid T^{(p+q+k)i}x \in C_{E_j} , 0 \leq i
		\leq \frac{n}{p+q+k}    \}|}{n/(p+q+k)} \\&= 1 
	\end{split}
	\end{equation}
	
	We know that
	$\sum_{j = 1}^{\infty} \gamma(C_{D_j}) + \sum_{j = 1}^{\infty}
	\gamma(C_{E_j}) = 1$.  And thus from equation (\ref{eqn:1}), we get
	that, $\sum_{j = 1}^{\infty} \gamma(C_{E_j}) = 1 - \gamma(C_{s'})$. 
	
	Repeating the procedure in the initial claim, for the second summand
	in equation (\ref{eqn:2}) we get that for every $\epsilon > 0$,
	there exist $M, M'$ such that for every $n \geq M'$
	\begin{align}
	\label{inequality3}
	\sum_{j = 1}^{M}
	\frac{|\{i \mid T^{(p+q+k)i}x \in C_{E_j} ,
		0\leq i\leq \frac{n}{p+q+k}\}|}
	{n/(p+q+k)}
	>
	1 - \gamma(C_{s'}) - \epsilon.
	\end{align} 
	Now, define
	\begin{align}
	H_1 =& \sum_{j =1}^{m} \frac{| \{i \mid T^{(p+q+k)i}x \in C_{D_j}, 0 \leq
		i \leq \frac{n}{p+q+k}    \}|}{n/(p+q+k)} \qquad\text{ and }\\
	T_1 =& \sum_{j =m+1}^{\infty} \frac{| \{i \mid T^{(p+q+k)i}x \in C_{D_j},
		0 \leq i \leq \frac{n}{p+q+k}    \}|}{n/(p+q+k)}.
	\end{align}
	Note that 
	$$ \sum_{j =1}^{\infty}
	\frac{| \{i \mid T^{(p+q+k)i}x \in C_{D_j},
		0 \leq i \leq \frac{n}{p+q+k}    \}|}
	{n/(p+q+k)} = H_1 + T_1.$$ 
	Similarly, define 
	\begin{align}
	H_2 =& \sum_{j =1}^{M} \frac{|\{i \mid T^{(p+q+k)i}x \in C_{E_j}, 0 \leq
		i \leq \frac{n}{p+q+k}    \}|}{n/(p+q+k)}\qquad\text{ and }\\
	T_2 =& \sum_{j =M+1}^{\infty} \frac{| \{i \mid T^{(p+q+k)i}x \in C_{E_j},
		0 \leq i \leq \frac{n}{p+q+k}    \}|}{n/(p+q+k)}.
	\end{align}
	We have,
	$$\sum_{j =1}^{\infty} \frac{|\{i \mid T^{(p+q+k)i}x \in C_{E_j}, 0 \leq
		i \leq \frac{n}{p+q+k}    \}|}{n/(p+q+k)} = H_2 + T_2.$$
	From equation (\ref{eqn:2}),
	$$ T_1 = 1 - H_1 - H_2 - T_2 $$

	Now, for every $n \geq m''$ where $m''$ = $\max \{m',M'\}$ we have,
	from the inequalities (\ref{inequality2}) and (\ref{inequality3})
	and since $T_{2}\geq 0$ we get,
	\begin{align}
	T_1 < 3\epsilon - T_2 < 3\epsilon.
	\end{align}
	To prove the main lemma it suffices to show that for every 
	$\epsilon > 0$, there is an $N$ such that for all $n \geq N$  
	$$ \Bigg|\sum_{j = 1}^{\infty} \frac{|\{i \mid T^{(p+q+k)i}x \in C_{D_j},
		0 \leq i \leq \frac{n}{p+q+k}    \}|}{n/(p+q+k)}
	- \gamma(C_{s'})\Bigg|
	< 5\epsilon.$$ 
	Now for every $n \geq m''$, 
	\begin{equation*}
	\begin{split}
	\Bigg|\sum_{j = 1}^{\infty} \frac{|\{i \mid T^{(p+q+k)i}x \in
		C_{D_j}, 0 \leq i \leq \frac{n}{p+q+k}    \}|}{n/(p+q+k)} -
	\gamma(C_{s'})\Bigg| &= |H_1 + T_1- \gamma(C_{s'})|\\ 
	&\leq |H_{1}-\gamma(C_{s'})|+T_{1}\\
	&< 2\epsilon + 3\epsilon\\
	&=5 \epsilon
	\end{split}
	\end{equation*}
	where the third statement above follows due to inequalities in
	(\ref{inequality2}) and (\ref{inequality4}). 
\end{proof}

The following is a generalization of Lemma \ref{lem:limitlemma},

\begin{lemma}
	\label{lem:generallimitlemma}
	Let $\{f_{j}(n)\}_{j=1}^{\infty}$ be a family of functions from $\mathbb{N} \to \mathbb{R}^{+}$ such that 
	\begin{align*}
	\forall j \in \mathbb{N}, \exists r_{j}>0, \quad \lim\limits_{n \to \infty} f_{j}(n)=r_{j}
	\end{align*}
	Now, let $\{g_{j}(n)\}_{j=1}^{\infty}$ be a family of functions from $\mathbb{N} \to \mathbb{R}$ such that 
	\begin{align*}
	\forall j \in \mathbb{N},\exists s_{j}>0, \quad \lim\limits_{n \to \infty} g_{j}(n)=s_{j}
	\end{align*}
	If there exists $c>0$ such that for all $n$,
	\begin{align*}
	\sum\limits_{j=1}^{\infty} f_{j}(n) + \sum\limits_{j=1}^{\infty} g_{j}(n) = c
	\end{align*}
	and if
	\begin{align*}
	\sum\limits_{j=1}^{\infty} r_{j} + \sum\limits_{j=1}^{\infty} s_{j} = c
	\end{align*}
	then,
	\begin{align*}
	\lim\limits_{n \to \infty}\sum\limits_{j=1}^{\infty} f_{j}(n)=\sum\limits_{j=1}^{\infty} r_{j}
	\end{align*}
\end{lemma}
A proof of this lemma can be obtained easily by generalizing the proof given above for Lemma \ref{lem:limitlemma}.

\section{The Converse to Pillai's Theorem}
In this section, we prove the following converse of the analogue of
Pillai's theorem for continued fractions.
\begin{theorem}
	\label{th:converseofpillai}
	Let $\mathnormal{x} \in (0,1)$ and $T$ be the Gauss map. If, for every
	$k$-length strings $s$ of positive integers,
	$$\lim_{n\to\infty} \frac{|\{i \mid T^{i}x \in C_s , 0 \leq i \leq n
		\}|} {n} = \gamma(C_s),$$
	then for any finite length string $t$ of positive integers,
	$$\lim_{n\to\infty} \frac{|\{i \mid T^{mi}x \in C_t , 0 \leq i \leq
		\frac{n}{m} \}|} {n/m} = \gamma(C_t),$$
	where $m$ is the length of the string $t$.
\end{theorem}
This can be obtained as a consequence of Theorem A in
\cite{moshchevitin} as we show in this section. We introduce the
notions relevant for our result from \cite{moshchevitin}.

Let $X$ be a measurable space with a $\sigma$-algebra $\mathcal{F}$
and probability measure $\mu$. Let $T:X\to X$ be ergodic with respect
to this measure. For $f: X \to \R$ which is measurable in this space
and $k \in \N$, consider the \emph{Birkhoff ergodic sum}
\begin{align}
S_k(f,x) =  \sum_{i=0}^{k-1} f(T^i x).
\end{align}
Suppose $\{C_m\}$ is a finite or countably infinite family of
measurable subsets of $X$.

Define the set-function $H$ defined for any set $E \subseteq X$ by
\begin{align}
\label{eqn:H}
H(E) = \inf \{ \mu(C_i) \mid \cup_{i} C_i \supseteq E \}.
\end{align}

\begin{definition}\cite{moshchevitin}
	We say that the measures $\mu$ and $H$ are \emph{co-ordinated} if
	any $\mu$-measurable set is $H$-measurable.
\end{definition}

\begin{definition}\cite{moshchevitin}
	Let $\{C_i\}$ be a finite or countably infinite family of
	$\mu$-measurable sets.  The family of \emph{$\{C_i\}$-approximable
		sets}, denoted $\Gamma(\{C_i\})$ is the family of $\mu$-measurable
	sets such that for any $\varepsilon > 0$, there exist disjoint
	collections $\{M_i\}$ and $\{N_i\}$ of sets from the family
	$\{C_i\}$ such that
	\begin{align}
	\sqcup M_i \subseteq V \subseteq \sqcup N_i
	\text{ and }
	\sum\mu(M_i)-\varepsilon < \mu(V) < \sum\mu(N_i)+\varepsilon.
	\end{align}
\end{definition}

We use the following result, due to Moshchevitin and Shkredov
(\cite{moshchevitin}, Theorem A), which is a generalization of the
theorem of Postnikov and Piatetskii-Shapiro
\cite{PostnikovPiatetskii57} (see also \cite{Bugeaud:UDMODA} Section
9.5).

\begin{theorem}\cite{moshchevitin}
	\label{thm:piatetskii}
	Let $(X,\mathcal{F},\mu)$ be a probability space, $x_0 \in X$, and
	$T:X \to X$ be ergodic. Let $\{C_m\}$ be a finite or countably
	infinite family of measurable subsets of $X$, and $H$ be a set
	function defined as in (\ref{eqn:H}). Suppose measures $\mu$ and $H$
	are co-ordinated and for any $A \in \mathcal{F}$, $\mu(A)=H(A)$.

	If there is a positive constant $C$ such that for any $I \in
	\{C_m\}$, the following inequality holds:
	
	\begin{align}
	\limsup_{n \to \infty} \frac{S_n(x_0, \chi_I)}{n} \le C \mu(I),
	\end{align}
	then for any set $I \in \Gamma(\{C_i\})$,
	\begin{align}
	\lim_{n \to \infty} \frac{S_n(x_0,\chi_I)}{n} = \mu(I).
	\end{align}
\end{theorem}

The proof of the theorem requires the fact that $T^{k}$ (where $T$ is
the Gauss transformation) is ergodic for any $k \geq 1$. It is known
that the Gauss transformation $T$ is strongly mixing (see
\cite{iosifescu}). Now the fact that $T^{k}$ is ergodic (in fact
strongly mixing) is immediate from the following lemma,
\begin{lemma}
	Let $S: (X,\mathcal{B},\mu)\to(X,\mathcal{B},\mu)$ be a measure
	preserving transformation which is strongly mixing. Then $S^{k}$ is
	strongly mixing for any $k \geq 1$
\end{lemma}
\begin{proof}
	Consider any two measurable sets $A,B$. 
	
	For any $k \geq 1$, the sequence $\{\mu(S^{-kn}(A)\cap
	B)\}_{n=0}^{\infty}$ is a subsequence of $\{\mu(S^{-n}(A)\cap
	B)\}_{n=0}^{\infty}$. 
	
	Since $S$ is strongly mixing, we have
	\begin{align*}
	\lim\limits_{n \to \infty}\mu(S^{-n}(A)\cap B) = \mu(A)\mu(B)
	\end{align*}
	which implies,
	\begin{align*}
	\lim\limits_{n \to \infty}\mu(S^{-kn}(A)\cap B) = \mu(A)\mu(B)
	\end{align*}
	Hence $S^{k}$ is a strongly mixing transformation.
\end{proof}
Now, we give the proof of Theorem \ref{th:converseofpillai}, which is
a consequence of the theorem due to Moshchevitin and Shkredov
\cite{moshchevitin}. This proof uses observations from the proof of converse to Pillai's theorem for normal numbers presented in Kuipers and Niederreiter \cite{niederreiter}.

\begin{proof}
	Let $\{E_m\}_{m=1}^{\infty}$ denote the family of cylinder sets over $[0,1]$ with
	positions $1$ to $l$ fixed for any $l \in \mathbb{N}$, defined
	by
	$$\{E_m\}_{m=1}^{\infty}=\{C_{w} \mid w \in \mathbb{N}^{*}\}.$$
	
	Now, as in \cite{moshchevitin}, let us define a set function $H(.)$
	that for any $S \subseteq [0,1]$, $H(S)=\inf\{ \sum \gamma(E_i)\}$,
	where $\inf$ is taken over coverings (finite or countable) of
	$S$. Let $\Gamma$ denote
	the family of $\gamma$-measurable sets $\{V\}$ that can be
	approximated with arbitrary accuracy by sets belonging to the family
	$ \{E_m \}$. It can be easily seen that by definition of $H$,
	$\gamma$ and $H$ are coordinated measures (as defined in
	\cite{moshchevitin}) and for any $\gamma$-measurable $S$,
	$\gamma(S)=H(S)$.
	
	Let $s$ be an arbitrary finite string of length $k$.
	
	$$\frac{|\{i \mid T^{ki}x \in C_s , 0 \leq i \leq \frac{n}{k} \}|}{n/k}
	\leq \frac{|\{i \mid T^{i}x \in C_s , 0 \leq i \leq n \}|}{n}k $$ By
	taking $\limsup$ on both sides and using the hypothesis, we get the
	following,
	\begin{align*}
	\limsup\limits_{n\rightarrow \infty} \frac{|\{i \mid T^{ki}x \in
		C_s , 0 \leq i \leq \frac{n}{k} \}|}{n/k} 
	&\leq \limsup\limits_{n\rightarrow \infty}\frac{|\{i \mid T^{i}x \in C_s
		, 0 \leq i \leq n \}|}{n}.k\\ 
	&= k.\gamma(C_s) 
	\end{align*}
	Now applying Theorem \ref{thm:piatetskii} in \cite{moshchevitin} for
	the transformation $T^{k}$ (which was shown to be ergodic at the
	start of this section) and using the observation above, we obtain
	that
	$$ \lim\limits_{n\rightarrow \infty} \frac{|\{i \mid T^{ki}x \in C_s
		, 0 \leq i \leq \frac{n}{k} \}|}{n/k} = \gamma(C_s). $$
\end{proof}

\section{Subsequence selection along arithmetic progressions
	violates normality}

In this section, we use Theorem \ref{th:converseofpillai} to obtain a
new, combinatorial proof of a recent result by Heersink and Vandehey
\cite{HeersinkVandehey16} that any subsequence selected along a
non-trivial arithmetic progression of indices from any continued
fraction normal results in a non-normal continued fraction.
\begin{theorem}[Heersink, Vandehey \cite{HeersinkVandehey16}]
	\label{thm:subsequence}
	Let $x = [x_1, x_2, \dots] \in [0,1]$ be a continued fraction
	normal. Let $k \ge 2$, and $b \ge 1$ be positive integers. Then the
	subsequence $[x_{b}, x_{b+k}, x_{b+2k}, \dots] \in [0,1]$ is not
	continued fraction normal.
\end{theorem}
The proof in Heersink and Vandehey employs an ergodic theoretic
argument by introducing a skew product, and establishing the result
from the ergodicity of the skew product and a property of the Gauss
measure.

The following proof does not employ ergodic theoretic methods, but
rather uses combinatorial methods to establish the same. We also show
that the property of the Gauss measure can be derived using a
combinatorial argument, thus yielding a new proof of the above
theorem.

The relevant property of the Gauss Measure is as follows.
\begin{lemma}[Heersink, Vandehey \cite{HeersinkVandehey16}]
	\label{lem:measure}
	Let $k \in \N_{\ge 2}$. Then
	$\gamma\left( C_{[1]} \cap T^{-k}\left(C_{[1]}\right)\right) >
	\gamma ( C_{[1,1]} )$.\footnote{This inequality has a minor error in
		\cite{HeersinkVandehey16}. The reader may verify the direction of
		the inequality by direct calculation for $k=2$. For the proof of
		the main theorem, \emph{viz.} Theorem \ref{thm:subsequence}, any
		\emph{strict} inequality suffices.}
\end{lemma}
We first assume the lemma above and prove Theorem
\ref{thm:subsequence} using the converse to Pillai's theorem for
continued fractions, \emph{i.e.} Theorem
\ref{th:converseofpillai}. Subsequently, we provide a proof of Lemma
\ref{lem:measure} which is substantially different from that in the
work of Heersink and Vandehey.

\begin{proof}[Proof of Theorem \ref{thm:subsequence}]
	Let $k \ge 2$ and $b \ge 1$ be positive integers, and consider the
	arithmetic progression of indices $b$, $b+k$, $b+2k$, $\dots$.
	
	Since normality does not depend on a finite prefix, we can see that
	a number $x$ is normal if and only if for any fixed $b \ge 1$,
	$T^{b} x$ is normal. Hence without loss of generality, we set $b=0$,
	in the subsequent argument.

	Let $x = [0; x_1, x_2, \dots]$ be a continued fraction normal. Then
	we show that the frequency of $[1,1]$ in the subsequence
	$x'=[0;x_k, x_{2k}, \dots]$ is
	$\gamma(C_{[1]}\cap T^{-k}(C_{[1]}))$. By Lemma \ref{lem:measure},
	this is not equal to the desired value of $\gamma(C_{[1,1]})$, thus
	establishing that the subsequence is non-normal.
	
	We proceed as follows. Consider the set of positions
	\begin{align*}
	S_n = \left\{ 1 \le i \le n \mid x'[i] = 1 \text{ and } x'[i+1] =
	1\right\}.
	\end{align*}
	
	Since $x'[i] = x[ki]$ for any $i \in \N$, the above set is
	\begin{align*}
	S_n =\left\{ 1 \le i \le n \mid x[ki] = 1 \text{ and } x[(i+1)k] =
	1\right\}.
	\end{align*}
	
	We can rewrite this as a disjoint union of blocks of length $2k$
	in the following manner. Each ``block'' contains two occurrences of
	1, spaced $k-1$ positions from each other. The intermediate places can be
	arbitrary positive integers. By making the block length to be $2k$
	and by making every block start at odd multiples of $k$, we ensure
	that the blocks are non-overlapping. This enables us to apply the
	converse of Pillai's theorem.
	\begin{multline*}
	S_n =
	\bigcup\limits_{n_1,\dots,n_{2(k-1)} \in \N^+}
	\{ 1 \le i \le n \mid
	i \text{ is odd},
	x[ik]=1,
	x[(i+1)k] = 1,
	\text{ and } \\
	x[ik+j]=n_j,
	x[(i+1)k+j] = n_{j+k-1},
	1 \le j \le k-1
	\}.
	\end{multline*}
	This expression makes it clear that the blocks of digits we consider
	are non-overlapping. 
	
	Our intention is to calculate $\lim_{n \to \infty} \frac{|S_n|}{n}$. Observe that,
	\begin{align*}
	\frac{|S_n|}{n} = \sum\limits_{j=1}^{\infty}\frac{|\{i \mid T^{2ki}x \in C_{G_j} ,
		0 \leq i \leq n\}|}
	{n}
	\end{align*}
	where $\langle G_1,G_2,G_3 \dots \rangle$ is any enumeration of $2k$-length strings such that the first and $(k+1)$\textsuperscript{th} positions are fixed to be 1.
	
	In order to calculate  $\lim_{n \to \infty} \frac{|S_n|}{n}$,  we
	use Lemma \ref{lem:generallimitlemma}. 
	
	Define,
	\begin{align*}
	f_{j}(n) = \frac{|\{i \mid T^{2ki}x \in C_{G_j} ,
		0 \leq i \leq n\}|}
	{n}
	\end{align*}
	From the converse to Pillai's theorem, for all $j$
	\begin{align*}
	\lim\limits_{n \to \infty}f_{j}(n) = \lim\limits_{n \to
		\infty}\frac{|\{i \mid T^{2ki}x \in C_{G_j} , 
		0 \leq i \leq n\}|}
	{n} = \gamma(C_{G_{j}})
	\end{align*}
	
	Now, define
	\begin{align*}
	g_{j}(n) = \frac{|\{i \mid T^{2ki}x \in C_{H_j} ,
		0 \leq i \leq n\}|}
	{n}
	\end{align*}
	where $\langle H_1,H_2,H_3 \dots \rangle$ is any enumeration of
	$(2k)$-length strings such that either the first or the
	$(k+1)$\textsuperscript{th} positions are not 1. 
	
	Again, from the converse to Pillai's theorem, for all $j$
	\begin{align*}
	\lim\limits_{n \to \infty}g_{j}(n) = \lim\limits_{n \to
		\infty}\frac{|\{i \mid T^{2ki}x \in C_{H_j} , 
		0 \leq i \leq n\}|}
	{n} = \gamma(C_{H_{j}})
	\end{align*}
	Now, observe that for all $n$,
	\begin{equation*}
	\label{eqn:3} 
	\begin{split}
	&\sum_{j =1}^{\infty} f_{j}(n) + \sum_{j =1}^{\infty} g_{j}(n) \\
	&=\sum_{j =1}^{\infty} \frac{|\{i \mid T^{2ki}x \in C_{G_j} ,
		0 \leq i \leq n\}|}
	{n} +  \sum_{j
		=1}^{\infty} \frac{|\{i \mid T^{2ki}x \in C_{H_j} ,
		0 \leq i \leq n\}|}
	{n} \\&= 1 
	\end{split}
	\end{equation*}
	and also,
	\begin{align*}
	\sum_{j = 1}^{\infty} \gamma(C_{G_j}) + \sum_{j = 1}^{\infty}
	\gamma(C_{H_j}) = 1
	\end{align*}
	Now,
	\begin{align*}
	\lim_{n \to \infty} \frac{|S_n|}{n}
	&=
	\lim_{n \to \infty} \sum\limits_{j=1}^{\infty}\frac{|\{i \mid
		T^{2ki}x \in C_{G_j} , 
		0 \leq i \leq n\}|}
	{n}\\
	&=
	\sum\limits_{j=1}^{\infty} \gamma(C_{G_{j}})\\
	&=
	\gamma(C_{[1]} \cap T^{-k}(C_{[1]})).
	\end{align*}
	Above, the second equality follows due to Lemma
	\ref{lem:generallimitlemma} and the third equality follows due to
	measure preservation. The frequency of the pattern in the
	subsequence is $\gamma(C_{[1]} \cap T^{-k}(C_{[1]}))$, which by
	Lemma \ref{lem:measure} is not equal to $\gamma(C_{[1,1]})$.
	
	Thus the asymptotic frequency of the pattern $[1,1]$ in $x'$ is not
	$\gamma(C_{[1,1]})$, proving that $x'$ is not continued-fraction
	normal.
\end{proof}

We now proceed to a proof of Lemma \ref{lem:measure}. The proof extant
in the literature relies on an earlier estimate by Wirsing
\cite{Wirsing74}, and uses inequalities related to transfer
operators. In the following proof, we employ a combinatorial approach,
analyzing into separate cases based on the values of the intermediate
co-ordinates. This part does not utilize Pillai's theorem or its
converse, but it is provided here as it completes the new
combinatorial proof of the result of Heersink and Vandehey.

The following two lemmas are easily established by induction.

\begin{lemma}
	\label{lem:q1_q3}
	Let $k \ge 1$, and $n_1$, $\dots$, $n_k$ be positive integers, with
	$n_k > 1$. Then the denominator of $[0;1,1,n_1,\dots,n_k]$ is
	greater than the denominator of $[0;1,n_1,\dots,n_k,1]$.
\end{lemma}
\begin{proof}
	Let $r$ be $[0;1,1,n_1,\dots,n_k]$ and $s$ be $[0;1,n_1,\dots,n_k]$.
	Note that
	$$[0;1,1,n_1,\dots,n_k] = \frac{1}{1+[0;1,n_1,\dots,n_k]}.$$ If $s$
	is the rational $p/q$ in lowest terms, where $p$ and $q$ are
	positive integers, then $r$ is $\frac{q}{p+q}$ in lowest
	terms. Hence the denominator of $r$ is greater than that of $s$, the
	difference being the numerator of $s$.
	
	We now analyze the difference between the denominators of $s$ and
	$[0;1,n_1,\dots,n_{k},1]$. By the continued fraction recurrence, the
	denominator of $[0;1,n_1,\dots,n_k,1]$ is larger than that of $s$ by
	the denominator of $[0;1,n_1,\dots,n_{k-1}]$.
	
	It suffices to show that the denominator of
	$[0;1,n_1,\dots,n_{k-1}]$ is less than $p$.
	
	By an argument similar to the one in preceding paragraphs,
	\begin{align*}
	\denominator([0;1,n_1,\dots,n_{k-1}]
	&= 
	\numerator([0;n_1,\dots,n_{k-1}]) +
	\denominator([0;n_1, \dots, n_{k-1}])\\
	&< 2 \times \denominator([0;n_1, \dots, n_{k-1}])\\
	&\le n_k \times \denominator([0;n_1,\dots,n_{k-1}]).
	\end{align*}
	Observing that the denominator of $[0;n_1,\dots,n_{k-1}]$ is the
	same as the numerator of $[0;1,n_1,\dots,n_{k-1}]$, we get that the
	quantity is less than
	$$n_k \times \text{numerator}([0;1,n_1,\dots,n_{k-1}]),$$
	By the continued fraction recurrence for numerators, this is less
	than $p$, as required.
\end{proof}

To handle cases where $n_k=1$, we use the following easily established
claim which relates that the Gauss measures of any cylinder and the
cylinder where the respective integer sequence is reversed.

\begin{lemma}
	\label{lem:reversecylinderlemma}
	Let $a_{1},a_{2},\dots a_{n} \in \mathbb{N}^+$. Then,
	\begin{align*}
	\gamma(C_{[a_{1},a_{2},\dots a_{n}]})=\gamma(C_{[a_{n},a_{n-1},\dots
		a_{1}]}) 
	\end{align*}
\end{lemma}
\begin{proof}
	Let us consider the case when $n$ is odd.
	
	Now,
	\begin{align*}
	\gamma\left(C_{[a_1,a_2\dots a_n]})\right)
	=
	\frac{1}{\ln 2}
	\int\limits_{[0;a_1,a_2\dots a_n+1]}^{[0;a_1,a_2\dots a_n]}
	\frac{1}{1+x}~dx
	=
	\log_2\left(
	\frac{1+[0;a_1,a_2\dots a_n]}{1+[0;a_1,a_2\dots a_n,1]}\right)
	\end{align*}
	and,
	\begin{align*}
	\gamma\left(C_{[a_{n},a_{n-1},\dots a_{1}]})\right)
	=
	\frac{1}{\ln 2}
	\int\limits_{[0;a_{n},a_{n-1},\dots a_{1}+1]}^{[0;a_{n},a_{n-1},\dots a_{1}]}
	\frac{1}{1+x}~dx
	=
	\log_2\left(
	\frac{1+[0;a_{n},a_{n-1},\dots a_{1}]}{1+[0;a_{n},a_{n-1},\dots
		a_{1},1]}\right) 
	\end{align*}
	Let,
	\begin{align*}
	[0;a_1,a_2\dots a_n] = \frac{p_n}{q_n}
	\end{align*}
	By the recurrence relation for convergents,
	\begin{align*}
	[0;a_1,a_2\dots a_n, 1] = \frac{p_n+p_{n-1}}{q_n+q_{n-1}}
	\end{align*}
	where $\frac{p_{n-1}}{q_{n-1}}=[0;a_{1},a_{2}\dots a_{n-1}]$.
	
	By the well-known result about convergents of reversed finite
	continued fractions (see Theorem 6 in \cite{Khinchin97}),
	\begin{align*}
	[0;a_{n},a_{n-1},\dots a_{1}]=\frac{q_{n-1}}{q_{n}}
	\end{align*}
	Since $[0;1,a_1,a_2\dots a_n]=\frac{q_{n}}{p_{n}+q_{n}}$ and
	$[0;1,a_1,a_2\dots a_{n-1}]=\frac{q_{n-1}}{p_{n-1}+q_{n-1}}$, by using
	the same fact about reversed finite continued fractions we get, 
	\begin{align*}
	[0;a_{n},a_{n-1},\dots a_{1}, 1]=\frac{p_{n-1}+q_{n-1}}{p_{n}+q_{n}}
	\end{align*}
	Now,
	\begin{align*}
	\gamma\left(C_{[a_1,a_2\dots a_n]})\right)
	=
	\log_2\left(
	\frac{1+\frac{p_n}{q_n}}{1+\frac{p_n+p_{n-1}}{q_n+q_{n-1}}}\right)
	=
	\log_2\left(
	\frac{(q_{n}+q_{n-1})(p_{n}+q_{n})}{q_{n}(p_n+p_{n-1}+q_n+q_{n-1})}\right)
	\end{align*}
	and,
	\begin{align*}
	\gamma\left(C_{[a_{n},a_{n-1},\dots a_{1}]})\right)
	=
	\log_2\left(
	\frac{1+\frac{q_{n-1}}{q_{n}}}{1+\frac{p_{n-1}+q_{n-1}}{p_{n}+q_{n}}}\right)
	=
	\log_2\left(
	\frac{(q_{n}+q_{n-1})(p_{n}+q_{n})}{q_{n}(p_n+p_{n-1}+q_n+q_{n-1})}\right)
	\end{align*}
	Hence, the lemma is true when $n$ is odd. The case when $n$ is even is
	similar. 
\end{proof}

We introduce some notation which is used in the proof of Lemma
\ref{lem:measure}.

{\bf Notation.} Let $p^1_{k+2}$ denote the numerator of the rational
$[0;1,1,n_1,\dots,n_k]$, and
$q^1_{k+2}$ denote its denominator. And let $p^3_{k+2}$ denote the numerator of the rational
$[0;1,n_1,\dots,n_k,1]$, and
$q^3_{k+2}$ denote its denominator.

We extend the notation naturally to refer to initial segments of
the respective cylinders as follows: for $1 \le j < k$,
\begin{align*}
\frac{p^1_{j+2}}{q^1_{j+2}} = [0;1,1,n_1,\dots,n_j],
\qquad\text{and }
\frac{p^3_{j+1}}{q^3_{j+1}} = [0;1,n_1,\dots,n_j], 
\end{align*}
where the fractions are expressed in lowest terms.

Finally, for the full cylinders, let 
\begin{align*}
\frac{p^1_{k+2}}{q^1_{k+2}} = [0;1,1,n_1,\dots,n_k], \qquad
\frac{p^2_{k+2}}{q^2_{k+2}} = [0;1,1,n_1,\dots,n_{k}+1],
\end{align*}
and
\begin{align*}
\frac{p^3_{k+2}}{q^3_{k+2}} = [0;1,n_1,\dots,n_k,1], \text{ and }
\frac{p^4_{k+2}}{q^4_{k+2}} = [0;1,n_1,\dots,n_k,2],
\end{align*}
where the fractions are expressed in lowest terms.

It is convenient to note that the extremities of the cylinder
$C_{[0;1,1,n_1,\dots,n_k]}$ correspond to the smaller superscripts $1$
and $2$, and the extremities of the cylinder
$C_{[0;1,n_1,\dots,n_k,1]}$ correspond to the superscripts $3$ and
$4$.

In this notation, Lemma \ref{lem:q1_q3} says that $q^1_{k+2} >
q^3_{k+2}$.

\begin{proof}[Proof of Lemma \ref{lem:measure}]
	Let $k \ge 1$. We prove the lemma by establishing that
	\begin{align}
	\label{ineq:cylinder_sum}
	\sum_{n_1, \dots, n_{k} \in \N^+}
	\gamma\left(C_{[1,n_1, \dots, n_{k},1]}\right)
	>
	\sum_{n_1, \dots, n_{k} \in \N^+}
	\gamma\left(C_{[1,1,n_1, \dots, n_{k}]}\right)
	\end{align}
	When $n_{k}=1$, the following holds as a consequence of Lemma
	\ref{lem:reversecylinderlemma},
	\begin{align*}
	\sum_{n_1, \dots, n_{k-1} \in \N^+}
	\gamma\left(C_{[1,n_1, \dots, n_{k-1},1,1]}\right)
	=
	\sum_{n_1, \dots, n_{k-1} \in \N^+}
	\gamma\left(C_{[1,1,n_1, \dots, n_{k-1},1]}\right)
	\end{align*}
	Hence it suffices to establish inequality (\ref{ineq:cylinder_sum})
	when $n_k>1$. We divide the proof into two cases, the first, when
	$k$ is even, and the second, when $k$ is odd.
	
	Assume $k$ is even. The Gauss measure of the cylinder
	$C_{[1,n_1,\dots,n_k,1]})$ is
	\begin{align}
	\label{ineq:measure_even_k:one}
	\gamma\left(C_{[1,n_1,\dots,n_k,1]})\right)
	=
	\frac{1}{\ln 2}
	\int\limits_{[0;1,n_1,\dots,n_k,1]}^{[0;1,n_1,\dots,n_k,2]}
	\frac{1}{1+x}~dx
	=
	\log_2\left(
	\frac{1+[0;1,n_1,\dots,n_k,2]}{1+[0;1,n_1,\dots,n_k,1]}\right)
	\end{align}
	and similarly,
	\begin{align}
	\label{ineq:measure_even_k:two}
	\gamma\left(C_{[1,1,n_1,\dots,n_k]})\right)
	=
	\log_2\left(
	\frac{1+[0;1,1,n_1,\dots,n_k+1]}{1+[0;1,1,n_1,\dots,n_k]}
	\right).
	\end{align}
	Since the logarithm is a monotone increasing function, it suffices
	to show that 
	\begin{align}
	\label{ineq:even_k_bare}
	\frac{1+[0;1,n_1,\dots,n_k,2]}{1+[0;1,n_1,\dots,n_k,1]}
	>
	\frac{1+[0;1,1,n_1,\dots,n_k+1]}{1+[0;1,1,n_1,\dots,n_k]}
	\end{align}
	

	The above is of the form $d/c > b/a$ where $d$ is the numerator on
	the left side, $c$ the denominator on the left, $b$ the numerator on
	the right and $a$ the denominator on the right. It is clear that
	$d>c$ and $b>a$. Let $b=a+\varepsilon$ and $d=c+\delta$, for
	positive rationals $\varepsilon$ and $\delta$.  To establish that
	$d/c > b/a$, it suffices to establish that
	$a(c+\delta)>c(a+\varepsilon)$, which reduces to $a
	(d-c)>c(b-a)$. We now establish this.
	
	We have
	$$a = 1+[0;1,1,n_1,\dots,n_k] =
	\frac{q^1_{k+2}+p^1_{k+2}}{q^1_{k+2}}.$$
	By observing that $[0;1,1,n_1,\dots,n_j] =
	\frac{1}{1+[0;1,n_1,\dots,n_j]}$, $1 \le j \le k$, we have 
	\begin{align}
	\label{eqn:q1_q3}
	p^1_{\ell} = q^{3}_{\ell-1}, \qquad 3 \le \ell \le k+2.
	\end{align}
	Hence
	\begin{align}
	\label{ineq:even_k_step_three_two}
	a &= \frac{q^1_{k+2}+q^{3}_{k+1}}{q^1_{k+2}}.
	\end{align}
	
	We now derive an identity for $c=1+[0;1,n_1,\dots,n_k,1]$. First, we
	establish the following relation.
	\begin{align}
	\label{eqn:q3_p3_q1}
	q^3_{j+1}+p^3_{j+1} = q^1_{j+2}, 
	\qquad 1 \le j \le k.
	\end{align}
	Consider the following fraction.
	\begin{align*}
	\frac{q^3_{j+1}+p^3_{j+1}}{q^3_{j+1}}
	=
	1+\frac{p^3_{j+1}}{q^3_{j+1}}
	= 1+[0;1,n_1, \dots,n_j].
	\end{align*}
	Now, note that
	\begin{align*}
	\frac{1}{1+[0;1,n_1, \dots,n_j]}
	=
	[0;1,1,n_1,\dots,n_j]
	=
	\frac{p^1_{j+2}}{q^1_{j+2}}.
	\end{align*}
	Since the numerators and denominators of the fractions above are in
	their reduced form, we obtain (\ref{eqn:q3_p3_q1}). Thus, we have
	the following expression for $c$.
	\begin{multline}
	\label{eqn:c_term}
	c = 1+ [0;1,n_1,\dots,n_k,1]
	= 1+ \frac{p^3_{k+2}}{q^3_{k+2}}
	= 1+\frac{p^3_{k+1}+p^3_{k}}
	{q^3_{k+1}+q^3_{k}}\\
	= \frac{q^3_{k+1}+p^3_{k+1} +
		q^3_{k}+p^3_{k}}
	{q^3_{k+1}+q^3_{k}}
	= \frac{q^1_{k+2} + q^1_{k+1}}{q^3_{k+1}+q^3_{k}}
	\end{multline}
	where the last equality follows from (\ref{eqn:q3_p3_q1}).
	
	We now have expressions for $a$ and $c$, and proceed to derive
	expressions for $(b-a)$ and $(d-c)$.
	
	Note that $(b-a)$ is the length (Lebesgue measure) of the cylinder
	$C_{[0;1,1,n_1, \dots, n_k]}$ and $(d-c)$, that of
	$C_{[0;1,n_1,\dots,n_k,1]}$. By elementary properties of continued
	fractions (see for example, Einsiedler and Ward
	\cite{einsiedler2013ergodic}, Chapter 3), we have
	\begin{align}
	\label{eqn:lengths_of_cylinders_two}
	(b-a) &= \frac{1}{q^1_{k+2}(q^1_{k+2}+q^1_{k+1})}\\
	(d-c) &= \frac{1} {q^3_{k+2}
		\left(q^3_{k+2}+q^3_{k+1}\right)}
	\end{align}
	Then, after simplifying the expression, we obtain that
	$a(d-c)>c(b-a)$ if and only if $q^1_{k+2} > q^3_{k+2}$. This is true
	by Lemma \ref{lem:q1_q3}. This establishes the lemma for the case
	when $k$ is even.
	
	When $k$ is odd, the expressions for the Gauss measure are as
	follows.
	\begin{align}
	\label{ineq:measure_odd_k:one}
	\gamma\left(C_{[1,n_1,\dots,n_k,1]}\right)
	=&
	\log_2\left(
	\frac{1+[0;1,n_1,\dots,n_k,1]}{1+[0;1,n_1,\dots,n_k,2]}\right),\\
	\gamma\left(C_{[1,1,n_1,\dots,n_k]}\right)
	=&
	\log_2\left(
	\frac{1+[0;1,1,n_1,\dots,n_k]}{1+[0;1,1,n_1,\dots,n_k+1]}
	\right).
	\end{align}
	Let $a$, $b$, $c$, and $d$ be defined as in the case for even $k$,
	To establish that the first cylinder has greater measure than the
	second, it suffices to show that $b(c-d)>d(a-b)$. We now have the
	following expressions for the lengths of the cylinders.
	\begin{align}
	\label{eqn:lengths_of_cylinders}
	(a-b) &= \frac{1}{q^1_{k+2}(q^1_{k+2}+q^1_{k+1})}\\
	(c-d) &= \frac{1} {q^3_{k+2}
		\left(q^3_{k+2}+q^3_{k+1}\right)}
	\end{align}
	Using (\ref{eqn:q1_q3}), we have the following expression for $b$.
	\begin{align*}
	b 
	= 1 + \frac{p^2_{k+2}}{q^2_{k+2}}
	= 1+\left(\frac{p^1_{k+2}+p^1_{k+1}}
	{q^1_{k+2}+q^1_{k+1}}\right)
	= \frac{q^1_{k+2}+q^1_{k+1}+q^3_{k+1}+q^3_{k}}
	{q^1_{k+2}+q^1_{k+1}}.
	\end{align*}
	Similarly, the variable $d$ has the following expression.
	\begin{align*}
	d 
	= 1+
	\left(\frac{2p^3_{k+1}+p^3_{k}}
	{2q^3_{k+1}+q^3_{k}}\right)
	= \frac{2q^1_{k+2}+q^1_{k+1}}{2q^3_{k+1}+q^3_{k}}.
	\end{align*}
	Then, we have
	\begin{align*}
	b(c-d)
	=&
	\frac{q^1_{k+2}+q^1_{k+1}+q^3_{k+1}+q^3_{k}}
	{q^1_{k+2}+q^1_{k+1}}
	\times
	\frac{1}{q^3_{k+2} (q^3_{k+2}+q^3_{k+1})}\\
	d(a-b)
	=&
	\frac{2q^1_{k+2}+q^1_{k+1}}{2q^3_{k+1}+q^3_k}
	\times
	\frac{1}{q^1_{k+2}(q^1_{k+2}+q^1_{k+1})}
	\end{align*}
	Observing that
	$$2q^3_{k+1}+q^3_k = q^3_{k+2}+q^3_{k+1},$$
	and canceling common terms, we obtain that $b(c-d)>d(a-b)$ if and
	only if
	\begin{align*}
	(q^1_{k+2}+q^1_{k+1}+q^3_{k+1}+q^3_k) q^1_{k+2}
	>
	(2 q^1_{k+2}+q^1_{k+1}) q^3_{k+2}.
	\end{align*}
	Subtracting $(q^3_{k+1}+q^3_k)q^1_{k+2}=q^3_{k+2}q^1_{k+2}$ from
	both sides, the inequality is true if and only if
	\begin{align*}
	(q^1_{k+2}+q^1_{k+1}) q^1_{k+2} > (q^1_{k+2}+q^1_{k+1}) q^3_{k+2},
	\end{align*}
	which is true if and only if $q^1_{k+2} > q^3_{k+2}$. The last
	inequality is true by Lemma \ref{lem:q1_q3}.
\end{proof}

\section*{Acknowledgements}
The authors gratefully acknowledge the support received from the Institute of Mathematical
Sciences, National University of Singapore. This work was partially
carried out during the workshop on ``Equidistribution: Computational
and Probabilistic Aspects'' at the IMS.  The authors wish to thank
Joseph Vandehey, Yann Bugeaud, and William Mance for helpful
suggestions and references to the Postnikov-Piatetskii Shapiro
criterion for continued fractions.

\bibliography{fair001}

\end{document}